\documentclass[12pt]{article}

\usepackage{graphics,amsmath,amssymb,amsthm,hyperref}

\graphicspath{{figures/}}

\newtheorem{theorem}{Theorem}
\newtheorem{prop}{Proposition}
\newtheorem{lemma}{Lemma}

\newcommand{\vcentergraphics}[1]{\ensuremath{\vcenter{\hbox{\includegraphics{#1}}}}}
\newcommand{\st}[2]{\ensuremath{\vcenter{\hbox{\scalebox{1}{\includegraphics{binary_tree_#1-#2}}}}}}
\newcommand{\arxiv}[1]{\href{http://arxiv.org/abs/#1}{\texttt{arXiv:#1}}}

\begin{document}

\begin{center}
\vskip 1cm{\LARGE\bf Non-Contiguous Pattern Avoidance in Binary Trees
}
\vskip 1cm
\large
Michael Dairyko\footnote{Partially supported by NSF grant DMS-0851721}\\
Department of Mathematics\\
Pomona College\\
Claremont, CA 91711, USA \\
\ \\
Lara Pudwell\footnotemark[1]  \\
Department of Mathematics and Computer Science\\
Valparaiso University\\
Valparaiso, IN 46383, USA\\
{\tt Lara.Pudwell@valpo.edu}\\
\ \\
Samantha Tyner\footnotemark[1]\\ 
Department of Mathematics\\
Augustana College\\
Rock Island, IL 61201, USA\\
\ \\
Casey Wynn\footnotemark[1]\\
Department of Mathematics and Computer Science \\
Hendrix College\\
Conway, AR 72032, USA
\end{center}

\begin{abstract}
In this paper we consider the enumeration of binary trees avoiding non-con\-tig\-uous binary tree patterns. We begin by computing closed formulas for the number of trees avoiding a single binary tree pattern with 4 or fewer leaves and compare these results to analogous work for contiguous tree patterns. Next, we give an explicit generating function that counts binary trees avoiding a single non-contiguous tree pattern according to number of leaves and show that there is exactly one Wilf class of $k$-leaf tree patterns for any positive integer $k$. In addition, we enumerate binary trees that simultaneously avoid more than one tree pattern. Finally, we explore connections between pattern-avoiding trees and pattern-avoiding permutations.
\end{abstract}

\section{Introduction}\label{S:Intro}

The notion of one object avoiding another has been studied in permutations, words, partitions, and graphs.  Much recent work has been devoted to the study of pattern-avoiding permutations.  Given permutations $\pi = \pi_1 \cdots \pi_n$ and $\rho=\rho_1 \cdots \rho_k$, we say that $\pi$ contains $\rho$ as a (classical) pattern if there exist indices $1 \leq i_1 < i_2 < \cdots < i_k \leq n$ such that $\pi_{i_1} \cdots \pi_{i_k}$ is order-isomorphic to $\rho$; that is, $\pi_{i_a} \leq \pi_{i_b}$ if and only if $\rho_a \leq \rho_b$.  Otherwise, $\pi$ is said to avoid $\rho$.  For example $\pi=24135$ contains the pattern $\rho=132$ as evidenced by $\pi_1=2$, $\pi_2=4$, and $\pi_4=3$, but $\pi$ avoids the pattern $321$ because $\pi$ has no decreasing subsequence of length 3.  One variation on this definition of pattern avoidance is to place the restriction $i_{j+1}=i_j + 1$ on the indices for $1 \leq j \leq k-1$.  If there exists such a subsequence of $\pi$ that is order-isomorphic to $\rho$, then $\pi$ is said to contain $\rho$ as a \emph{consecutive} permutation pattern.  For each of these definitions, two natural questions arise: ``Given a permutation $\rho$, how many permutations of length $n$ avoid $\rho$?'' and ``When do two distinct permutations $\rho_1$ and $\rho_2$ yield the same avoidance generating function?''  Patterns $\rho_1$ and $\rho_2$ with this property are said to be \emph{Wilf-equivalent}.

In this paper we consider analogous questions of pattern avoidance for plane trees. All trees in the paper are rooted and ordered.  We will focus on full binary trees, that is, trees in which each vertex has 0 or 2 (ordered) children.  Two children with a common parent are \emph{sibling vertices}.  A vertex with no children is a \emph{leaf} and a vertex with 2 children is an \emph{internal vertex}.  A binary tree with $n$ leaves has $n-1$ internal vertices, and the number of such trees is $\frac{\binom{2n-2}{n-1}}{n}$ (OEIS A000108).  For simplicity of computation, we adopt the convention that there are zero rooted binary trees with zero leaves.  The first few binary trees are depicted in Figure \ref{F:trees}.

Conceptually, a plane tree $T$ avoids a tree pattern $t$ if there is no instance of $t$ anywhere inside $T$. Pattern avoidance in vertex-labeled trees has been studied in various contexts by Steyaert and Flajolet \cite{SF83}, Flajolet, Sipala, and Steyaert \cite{FSS90}, Flajolet and Sedgewick \cite{FS09}, and Dotsenko \cite{DTBA}.  Recently, Disanto \cite{DiTBA} studied pattern containment of caterpillar subgraphs in binary trees while Khoroshkin and Piontkovski \cite{KPTBA} considered generating functions for general unlabeled tree patterns in a different context.

In 2010, Rowland \cite{Rowland10} explored contiguous pattern avoidance in binary trees (that is, rooted ordered trees in which each vertex has 0 or 2 children).  He chose to work with binary trees because there is natural bijection between $n$-leaf binary trees and $n$-vertex trees.  His study had two main objectives.  First, he developed an algorithm to find the generating function for the number of $n$-leaf binary trees avoiding a given tree pattern; he adapted this to count the number of occurrences of the given pattern.  Second, he determined equivalence classes for binary tree patterns, classifying two trees $s$ and $t$ as equivalent if the same number of $n$-leaf binary trees avoid $s$ as avoid $t$ for $n \geq 1$.  He completed the classification for all binary trees with at most eight leaves, using these classes to develop replacement bijections between equivalent binary trees.

In 2012, Gabriel, Peske, Tay, and the second author \cite{VERUM10} considered Rowland's definition of tree pattern in ternary, and more generally in $m$-ary, trees.  After generalizing Rowland's algorithmic approach to compute functional equations for the avoidance generating functions of arbitrary ternary tree patterns, they explored bijections between equinumerous sets of pattern-avoiding trees.  Along the way they found sets of pattern-avoiding trees whose enumeration yielded a number of well-known sequences.

In this paper, we extend Rowland's work in a new direction.  The work of \cite{VERUM10} and \cite{Rowland10} may be seen as parallel to the definition of consecutive permutation patterns given at the beginning of this section.  In those papers, tree $T$ was said to contain tree $t$ as a (contiguous) pattern if $t$ was a contiguous, rooted, ordered, subtree of $T$.  In this paper, we modify the definition of tree pattern to mirror the idea of classical pattern avoidance in permutations.  In the case of pattern-avoiding permutations, there are more Wilf-classes for consecutive patterns of a given length than for classical patterns (as a small example, there are 7 Wilf classes of consecutive permutation patterns of length 4, compared to 3 Wilf classes for classical permutation patterns of length 4).  This parallel holds true in the case of trees.  In fact, as we show in Section \ref{S:gfs}, there is precisely one Wilf class of $k$-leaf patterns for any $k \in \mathbb{Z}^+$.

As with previous work, given any binary tree pattern $t$ we present a technique to compute the generating function that enumerates trees avoiding $t$ according to the number of leaves.  This enumeration shows that there is exactly one Wilf class of $k$-leaf patterns.  We also consider trees avoiding multiple tree patterns and explore relationships between sets of pattern-avoiding trees and pattern-avoiding permutations.  

\section{Definitions and Notation}\label{S:Defs}

In this paper, a tree $T$ \emph{contains} $t$ as a (non-contiguous) tree pattern if there is a tree $T^*$, obtained from $T$ via a finite sequence of edge contractions, such that $t$ is a contiguous, rooted, and ordered subtree of $T^*$.  Conversely, $T$ \emph{avoids} $t$ if there is no such $T^*$ that contains $t$ as a subtree.  For example, consider the three trees shown in Figure \ref{Feg1}.  $T$ avoids $t$ as a contiguous pattern, but $T$ contains $t$ non-contiguously  (contract all non-bolded edges).  On the other hand, $T$ avoids $s$ both contiguously and non-contiguously since no vertex of $T$ has a left child and a right child, both of which are internal vertices.

\begin{figure}[htb]
\begin{center}
$T=$  \vcentergraphics{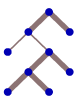} \hspace{.4in}
$t=$ \vcentergraphics{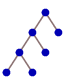}\hspace{.4in}
$s=$ \vcentergraphics{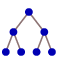}
\end{center}
\caption{Three binary trees}
\label{Feg1}
\end{figure}

We define $\text{Av}_t(n)$ to be the set of $n$-leaf binary trees that avoid the pattern $t$ non-contiguously, and $\text{av}_t(n)= \left| \text{Av}_t(n) \right|$.  We will also write $\text{Av}_t^c(n)$ for the set of $n$-leaf binary trees that avoid $t$ contiguously, and $\text{av}_t^c(n) = \left| \text{Av}_t^c(n)\right|$.  We will be particularly interested in determining the generating function $$\displaystyle{g_t(x)=\sum_{n=0}^\infty \text{av}_t(n) x^n}$$ for various tree patterns $t$.

Before we explore particular binary tree patterns, we list all of the 1, 2, 3, and 4-leaf binary trees.  We label trees with a double subscript notation.  The first subscript gives the number of leaves of the tree, and the second subscript distinguishes between distinct tree patterns with the same number of leaves.  We also note that if $t^r$ is the left-right reflection of tree $t$, then $\text{av}_t(n)=\text{av}_{t^r}(n)$ by symmetry, so we omit left-right reflections.  We will use these labels throughout the remainder of the paper.

\begin{figure}
\begin{center}
$t_{1_1}=$  \vcentergraphics{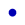} \hspace{.4in}
$t_{2_1}=$  \vcentergraphics{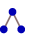} \hspace{.4in}
$t_{3_1}=$  \vcentergraphics{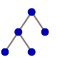}  \\ \vspace{3mm}
$t_{4_1}=$  \vcentergraphics{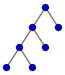} \hspace{.4in}
$t_{4_2}=$  \vcentergraphics{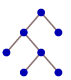} \hspace{.4in}
$t_{4_3}=$  \vcentergraphics{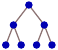} 
\end{center}
\caption{Binary trees with less than 5 leaves}
\label{F:trees}
\end{figure}

\section{Avoiding Simple Tree Patterns}\label{S:simple}

In this section, we find generating functions for the number of trees avoiding each of the tree patterns in Figure \ref{F:trees}.  For each tree, we discuss the structure of trees that avoid the given tree pattern and explain how to find a recurrence and generating function from this structure.  We also compare these generating functions to previously known results for contiguous tree patterns. 

\subsection{Avoiding \texorpdfstring{$t_{1_1}$}{t11}, \texorpdfstring{$t_{2_1}$}{t21}, and \texorpdfstring{$t_{3_1}$}{t31}}

Clearly, a tree avoids $t_{1_1}$ if and only if it has no vertices.  This is also true for contiguous avoidance, so we have

$$\text{av}_{t_{1_1}}(n)=\text{av}_{t_{1_1}}^c(n)=0 \text{ for } n\geq 1 \hspace{.5in} \text{and} \hspace{.5in} g_{t_{1_1}}(x)=0.$$

Next, a tree avoids $t_{2_1}$ if and only if it has no vertex with two children.  In other words, only $t_{1_1}$ avoids $t_{2_1}$.  Again, we have

$$\text{av}_{t_{2_1}}(n)=\text{av}_{t_{2_1}}^c(n)=\begin{cases}1&n=1\\0&n>1\end{cases} \hspace{.5in} \text{and} \hspace{.5in} g_{t_{1_1}}(x)=x.$$

Finally, we observe that tree $t$ avoids $t_{3_1}$ if and only if $t$ has no vertex whose left child is not a leaf.  For each $n\geq 1$ there is precisely one such tree, so we have

$$\text{av}_{t_{3_1}}(n)=\text{av}_{t_{3_1}}^c(n)=1 \text{ for } n\geq 1 \hspace{.5in} \text{and} \hspace{.5in} g_{t_{3_1}}(x)=\frac{x}{1-x}.$$

From these few cases, it may seem that non-contiguous and contiguous avoidance is the same for many trees.  The reader should suspect that this similarity does not hold for larger tree patterns based on the example of Figure \ref{Feg1}.  This suspicion is confirmed when we consider 4-leaf tree patterns.

Rowland showed that for $n> 1$
$$\text{av}_{t_{4_1}}^c(n) = M_{n-1} \hspace{.5in} \text{and} \hspace{.5in} \text{av}_{t_{4_2}}^c(n)=\text{av}_{t_{4_3}}^c(n)=2^{n-2}$$ where $M_n$ is the $n$th Motzkin number (OEIS A001006).

It turns out that non-contiguous avoidance is even more well-behaved for 4-leaf tree patterns.  We will show that for $n > 1$

$$\text{av}_{t_{4_1}}(n) =\text{av}_{t_{4_2}}(n)=\text{av}_{t_{4_3}}(n)=2^{n-2}.$$

\subsection{Avoiding \texorpdfstring{$t_{4_1}$}{t41}}

To find $\text{av}_{t_{4_1}}(n)$, we consider the structure of a general $n$-leaf binary tree $T$ that avoids $t_{4_1}$.  We have two cases.  Let $v$ be the root of $T$.  In the first case, $v$'s left child has no children, while $v$'s right child is the root of an $(n-1)$-leaf subtree avoiding $t_{4_1}$.  In the second case, $v$'s left child has two children, but the leftmost of these children is a leaf.  If the right child of $v$'s left child is the root of a subtree with $i$ leaves ($1 \leq i \leq n-2$), then $v$'s right child is the root of a subtree with $n-i-1$ leaves.  Further, the $i$-leaf subtree must avoid the tree pattern $t_{3_1}$.

In the first case, we considered $\text{av}_{t_{4_1}}(n-1)$ possible trees.  In the second case, we considered $\displaystyle{\sum_{i=1}^{n-2} \text{av}_{t_{3_1}}(i) \text{av}_{t_{4_1}}(n-i-1)}$ trees. However, since we know that $\text{av}_{t_{3_1}}(i) = 1$ for $i \geq 1$, we have

$$\text{av}_{t_{4_1}}(n) = \text{av}_{t_{4_1}}(n-1) + \sum_{i=1}^{n-2} \text{av}_{t_{4_1}}(n-i-1) = \sum_{i=0}^{n-2} \text{av}_{t_{4_1}}(n-i-1)=\sum_{i=1}^{n-1}\text{av}_{t_{4_1}}(n-i).$$

Together with the base case $\text{av}_{t_{4_1}}(1)=\text{av}_{t_{4_1}}(2)=1$, we see that

$$\text{av}_{t_{4_1}}(n)=2^{n-2} \text{ for } n>1 \hspace{.5in} \text{and} \hspace{.5in} g_{t_{4_1}}(x)=\frac{x-x^2}{1-2x}.$$

\subsection{Avoiding \texorpdfstring{$t_{4_2}$}{t42}}

Next, we consider the structure of a general $n$-leaf binary tree $T$ that avoids $t_{4_2}$ where $v$ is the root of $T$.  Again, we have two cases.  As before, in the first case, $v$'s left child has no children, while $v$'s right child is the root of an $(n-1)$-leaf subtree avoiding $t_{4_2}$.  In the second case, $v$'s left child has two children, but now the rightmost of these children is a leaf. The subtree whose root is the left child of $v$'s left child must avoid the tree pattern $t_{3_1}^r$.  After a nearly identical calculation to that of avoiding $t_{4_1}$, we see that

$$\text{av}_{t_{4_2}}(n)=2^{n-2} \text{ for } n>1 \hspace{.5in} \text{and} \hspace{.5in} g_{t_{4_2}}(x)=\frac{x-x^2}{1-2x}.$$

\subsection{Avoiding \texorpdfstring{$t_{4_3}$}{t43}}

Finally, we consider the structure of a general $n$-leaf binary tree $T$ that avoids $t_{4_3}$ where $v$ is the root of $T$.  Again, we have two cases: either $v$'s left child has children or $v$'s right child has children, but not both.  There are $\text{av}_{t_{4_3}}(n-1)$ trees that fall into the first case, and $\text{av}_{t_{4_3}}(n-1)$ trees that fall into the second case, which yields

$$\text{av}_{t_{4_3}}(n) = 2\text{av}_{t_{4_3}}(n-1).$$

We have
$$ g_{t_{4_3}}(x)=\frac{x-x^2}{1-2x}.$$

\section{Generating Functions for Pattern-Avoiding Trees}\label{S:gfs}

We have just seen that all 4-leaf binary tree patterns produce the same avoidance sequence when considered as non-contiguous patterns.  This observation leads naturally to the main theorem of this paper, which provides a particularly clean answer to both questions stated in the introduction.  Namely, ``given a tree pattern $t$, how many trees with $n$ leaves avoid $t$?'' and ``given two distinct tree patterns $t$ and $s$, when do $t$ and $s$ produce the same avoidance sequence?''  In fact, we see that not only is exact enumeration possible for \emph{any} non-contiguous tree pattern, we see that all avoidance generating functions are rational and of a particularly attractive form.

\begin{theorem}
Let $k \in \mathbb{Z}^{+}$ and let $t$ be a binary tree pattern with $k$ leaves.  Then 

$$g_{t}(x) = \frac{\sum\limits_{i=0}^{\lfloor \frac{k-2}{2} \rfloor} (-1)^{i}\cdot \binom{k-(i+2)}{i}\cdot x^{i+1}}{\sum\limits_{i=0}^{\lfloor \frac{k-1}{2} \rfloor}(-1)^{i}\cdot \binom{k-(i+1)}{i}\cdot x^{i}}.$$
\label{Th1}
\end{theorem}

The reader can check that this indeed matches the generating functions computed in the previous section when $k \leq 4$.  Further, notice that the generating function given in Theorem \ref{Th1} depends only on the number of leaves of $t$, and not on $t$ itself; that is, there is precisely one Wilf class of $k$-leaf tree patterns for each $k \in \mathbb{Z}^+$.  As we will see in Section \ref{S:multipatterns} and the Appendix, the analogous statement for pairs of trees is not true.  In this section we build the necessary framework to prove Theorem \ref{Th1}.

First, following \cite{VERUM10,Rowland10} we say that $T$ contains pattern $p$ at the root if $T$ contains a \emph{contiguous} copy of $p$ where the root of $p$ coincides with the root of $T$.  Now, let $g_{(t;p)}(x)$ be the generating function that enumerates binary trees avoiding tree pattern $t$ non-contiguously and containing the contiguous tree pattern $p$ at their root according to number of leaves.  Because all binary trees have a root vertex, it follows that the generating function for all trees avoiding $t$ is given by $g_{t}(x) = g_{(t;\st{1}{1})}(x)$.  Also, let $t_{\ell}$ and $t_{r}$ denote the subtrees descending from the left and right children of the root of $t$ respectively.

Since we are working with full binary trees, the root has either zero or two children. When there are zero children, we have a 1-leaf tree.  When there are two children, we have a tree with pattern \st{2}{1} at the root.  Such trees are enumerated with the generating function $g_{(t;\st{2}{1})}(x)$. Thus, we have

\begin{equation}
g_{(t;\st{1}{1})}(x)= x + g_{(t;\st{2}{1})}(x).
\label{eq:t11}
\end{equation}

Next, we determine a recurrence for $g_{(t;\st{2}{1})}(x)$.  Consider a tree $T$ that avoids $t$ and has the contiguous pattern \st{2}{1} at the root.  Either the subtree extending from the left child of the root of $T$ avoids $t_{\ell}$, the subtree extending from the right child of the root of $T$ avoids $t_r$, or both.  Inclusion-exclusion gives
 
\begin{equation}
g_{(t;\st{2}{1})}(x)= g_{(t_{\ell};\st{1}{1})}(x) \cdot g_{(t;\st{1}{1})}(x) + g_{(t;\st{1}{1})}(x) \cdot g_{(t_r;\st{1}{1})}(x) - g_{(t_{\ell};\st{1}{1})}(x) \cdot g_{(t_r;\st{1}{1})}(x).
\label{eq:t21}
\end{equation}

Combining Equations \ref{eq:t11} and \ref{eq:t21} yields

$$g_{t}(x) = x + g_{t_{\ell}}(x) \cdot g_{t}(x) + g_{t}(x) \cdot g_{t_r}(x) - g_{t_{\ell}}(x) \cdot g_{t_r}(x).$$   

Now, solve for $g_{t}(x)$ to obtain

\begin{equation}
g_{t}(x)= \frac{x - g_{t_{\ell}}(x) \cdot g_{t_r}(x)}{1-g_{t_{\ell}}(x) - g_{t_r}(x)}.
\label{eq:genfunrec}
\end{equation}

This computation yields one immediate result:

\begin{prop}
For any tree pattern $t$, $g_{t}(x)$ is a rational function of $x$. 
\end{prop}

\begin{proof}
We have already seen that $g_{t}(x)$ is rational for all tree patterns $t$ with $k \leq 4$ leaves.  Thus, by induction and Equation \ref{eq:genfunrec}, the proposition follows.
\end{proof}

Equation \ref{eq:genfunrec} simplifies further for one particular family of binary tree patterns.  Let the $k$-leaf left comb be the unique $k$-leaf binary tree where every right child is a leaf.  Write $c_k$ for the $k$-leaf left comb.  Then $c_1=t_{1_1}$, $c_2=t_{2_1}$, $c_3=t_{3_1}$, and $c_4=t_{4_1}$.  

\begin{lemma}
$g_{c_k}(x)=\frac{x}{1-g_{c_{k-1}}(x)}$ for $k \geq 2$.
\end{lemma}

\begin{proof}

Let $t=c_k$ in Equation \ref{eq:genfunrec}.  Then $t_{\ell}=c_{k-1}$ and $t_r=\st{1}{1}$.  We have:

$$g_{c_k}(x)= \frac{x - g_{c_{k-1}}(x) \cdot g_{\st{1}{1}}(x)}{1-g_{c_{k-1}}(x) - g_{\st{1}{1}}(x)}.$$

Since $g_{\st{1}{1}}(x)= 0$, this simplifies to

$$g_{c_k}(x)= \frac{x}{1-g_{c_{k-1}}(x)}. \mbox{\qedhere}$$
\end{proof}

This nice relationship between $g_{c_k}(x)$ and $g_{c_{k-1}}(x)$ seems natural because of the large overlap between copies of $c_k$ and $c_{k-1}$.  Our next lemma shows that there is a simple relationship between generating functions for non-comb tree patterns as well.

\begin{lemma}
Fix $k \in \mathbb{Z}^+$.  Let $t$ and $s$ be two $k$-leaf binary tree patterns.  Then

$$g_{t}(x)=g_{s}(x).$$
\end{lemma}

\begin{proof}
Assume the lemma holds for tree patterns with $\ell$ leaves where $\ell < k$.  Since we suppose that all $\ell$-leaf trees have the same avoidance generating function, we have that any $\ell$-leaf tree has avoidance generating function $g_{c_\ell}(x)$.  Suppose that tree $t^*$ has $\ell$ leaves to the left of its root and $k-\ell$ leaves to the right of its root.  Similarly, suppose that tree $s^*$ has $\ell+1$ leaves to the left of its root and $k-\ell-1$ leaves to the right of its root.  We will show that $g_{t^*}(x)=g_{s^*}(x)$.  

From Equation \ref{eq:genfunrec}, we have

\begin{align}
\begin{split}
g_{t^*}(x)&= \frac{x-g_{c_\ell}(x)\cdot g_{c_{k-\ell}}(x)}{1-g_{c_\ell}(x)-g_{c_{k-\ell}}(x)} \\
&= \frac{x-g_{c_\ell}(x)\cdot \left(\frac{x}{1-g_{c_{k-\ell-1}}(x)}\right)}{1-g_{c_\ell}(x)-\left(\frac{x}{1-g_{c_{k-\ell-1}}(x)}\right)} \\
&= \frac{x\left(-1+g_{c_\ell}(x)+g_{c_{k-\ell-1}}(x)\right)}{1-x-g_{c_\ell}(x)-g_{c_{k-\ell-1}}(x)+g_{c_\ell}(x)\cdot g_{c_{k-\ell-1}}(x)}.  \\
\end{split}
\end{align}

Similarly,

\begin{align}
\begin{split}
g_{s^*}(x)&= \frac{x-g_{c_{\ell+1}}(x)\cdot g_{c_{k-\ell-1}}(x)}{1-g_{c_{\ell+1}}(x)-g_{c_{k-\ell-1}}(x)} \\
&= \frac{x-\left(\frac{x}{1-g_{c_{\ell}}(x)}\right)\cdot g_{c_{k-\ell-1}}(x)}{1-\left(\frac{x}{1-g_{c_{\ell}}(x)}\right)-g_{c_{k-\ell-1}}(x)} \\
&= \frac{x\left(-1+g_{c_\ell}(x)+g_{c_{k-\ell-1}}(x)\right)}{1-x-g_{c_\ell}(x)-g_{c_{k-\ell-1}}(x)+g_{c_\ell}(x)\cdot g_{c_{k-\ell-1}}(x)}. \\
\end{split}
\end{align}

Thus 

$$g_{t^*}(x)=g_{s^*}(x).$$

Since this holds for $1 \leq \ell \leq k-1$, by transitivity, all $k$-leaf tree patterns have the same non-contiguous avoidance generating function. 
\end{proof}

Finally, to prove the main theorem, it suffices to show that

$$g_{c_k}(x)=\frac{\sum\limits_{i=0}^{\lfloor \frac{k-2}{2}\rfloor} (-1)^{i}\cdot \binom{k-(i+2)}{i}\cdot x^{i+1}}{\sum\limits_{i=0}^{\lfloor \frac{k-1}{2}\rfloor}(-1)^{i}\cdot \binom{k-(i+1)}{i}\cdot x^{i}}.$$

It is a straightforward induction proof, using the fact that $g_{c_k}(x)= \frac{x}{1-g_{c_{k-1}}(x)}$ to show that this formula holds in general.

We have now explicitly enumerated trees avoiding any single binary tree pattern non-contiguously and determined all equivalences for when two trees exhibit the same avoidance sequence.  We display the explicit generating function, first 8 sequence terms, and appropriate OEIS entry for tree patterns with $k$ leaves where $k \leq 9$ in Table \ref{tableseqs}.

\begin{table}
\begin{center}
\begin{tabular} {|c|cccc|}
\hline
$k$ & $g_{c_k}(x)$ & Sequence &  Growth rate& OEIS number\\
\hline
1 &  0 & $0, 0, 0, 0, 0, 0, 0, 0 \dots$ &0& trivial\\
\hline  
2 &  $x$& $1, 0, 0, 0, 0, 0, 0, 0,  \dots$ & 0& trivial\\
\hline  
3 &  $\frac{x}{1-x} $& $1, 1, 1, 1, 1, 1, 1, 1,  \dots$ & $1^n$ & trivial\\
\hline  
4 &  $\frac{x-x^2}{1-2x} $& $1, 1, 2, 4, 8, 16, 32, 64, \dots$ & $2^n$& A000079\\
\hline  
5& $\frac{x-2x^2}{1- 3x + x^2}$  & $1, 1, 2, 5, 13, 34, 89, 233, \dots$ &$\left(\frac{3+\sqrt{5}}{2}\right)^n$& A001519 \\
\hline
6& $\frac{ x -3x^{2}+ x^{3}}{1- 4x + 3x^{2}}$  & $1, 1, 2, 5, 14, 41, 122, 365,  \dots$ & $3^n$&  A007051  \\
\hline
7& $\frac{x -4x^{2}+3x^3}{1-5x +6x^{2} - x^{3}}$& $1, 1, 2, 5, 14, 42, 131, 417, \dots$ &$\approx(3.247)^n$& A080937  \\
\hline
8 & $\frac{x -5x^{2}+6x^{3}- x^{4}} {1- 6x +10x^{2} - 4x^{3}}$& $1, 1, 2, 5, 14, 42, 132, 428, \dots$ &$(2+\sqrt{2})^n$& A024175 \\
\hline
9 & $\frac{x -6x^{2} +10x^{3} -4x^{4}}{1-7x+15x^{2} -10x^{3} +x^{4}}$& $1, 1, 2, 5, 14, 42, 132, 429, \dots$ &$\approx(3.532)^n$& A080938\\
\hline
\end{tabular}
\end{center}
\caption{Enumeration data for binary tree patterns with $k \leq 9$ leaves}
\label{tableseqs}
\end{table}

Because all generating functions $g_{t}(x)$ are rational it follows that for any $k$-leaf binary tree $t$ the sequence $\{\text{av}_{t}(n)\}_{n=1}^\infty$ satisfies a linear recurrence with constant coefficients.  In fact, when tree pattern $t$ has $k \geq 4$ leaves, $\{\text{av}_{t}(n)\}_{n=1}^\infty$ grows exponentially.  Because there are Catalan-many binary trees with $n$ leaves, the growth of these sequences is bounded above by $4^n$.  In Table \ref{tableseqs} we display also the asymptotic growth rate for $3 \leq k \leq 9$.

\section{Recurrences for Binary Trees}\label{S:bijections}

While we gave explicit combinatorial explanations for the recurrences satisfied by $\text{av}_t(n)$ when $t$ has $k=3$ or $k=4$ leaves in Section \ref{S:simple}, the rest of our work has been algebraic.   It is possible, however, to derive the recurrences satisfied by tree enumeration sequences through other techniques.  In this section, we give a combinatorial explanation for these recurrences by considering the structure of trees that avoid the $k$-leaf left comb.

First, we define a many-to-one correspondence between the set of $(n+1)$-leaf binary trees and the set of  $n$-leaf binary trees.  Given a tree $t$ with $n+1$ leaves, let $v_r$ be rightmost leaf whose sibling vertex is also a leaf.  Then, let the \emph{parent tree} of $t$ be the $n$-leaf tree obtained by deleting $v_r$ and its sibling.  If $\hat{t}$ is the parent tree of $t$, then we say that $t$ is a \emph{descendent tree} of $\hat{t}$.  While the parent tree of a tree is unique, a given tree may have multiple descendent trees.  For example, $\st{4}{1}$, $\st{4}{2}$, and $\st{4}{3}$ are descendent trees of $\st{3}{1}$.  Similarly, $\st{4}{4}$ and $\st{4}{5}$ are descendent trees of $\st{3}{2}$.

Now, for any $n$-leaf tree $t_n$, we can use this parent/descendent relationship to generate a list $t_1, t_2, t_3, \dots, t_n$ where $t_1 = \st{1}{1}$, $t_i$ has $i$ leaves for $1 \leq i \leq n$ and $t_{i+1}$ is a descendant tree of $t_i$.  We refer to such a sequence of trees as the \emph{ancestry} of $t$.  Because parent trees are unique, the ancestry of tree $t$ must be unique.  For example, the tree \st{5}{7} has ancestry

\begin{center}
$\st{1}{1} \to \st{2}{1} \to \st{3}{1} \to \st{4}{2} \to \st{5}{7}$
\end{center}

Given an $n$-leaf tree $t$, we may also generate all descendent trees of $t$ in a systematic way.  Call an internal vertex $v$ of $t$ \emph{closed} if $v$'s right child is not a leaf.  The number of descendent trees of $t$ is equal to the number of leaves of $t$ that appear to the right of all closed vertices.  In fact, to produce a descendent tree of $t$, we need only attach a pair of children to any one such leaf.  The descendent tree relationship is further articulated in the following proposition.

\begin{prop}
Suppose that tree $t$ has $i$ descendant trees.  Then those descendant trees will have $2, 3, 4, \dots , i+1$ descendant trees respectively.  Further, if $t$ has $i$ descendant trees then the leftmost vertex to which one may add children is the leftmost vertex in a copy of an $i$-leaf left comb.
\label{prop:descs}
\end{prop}

\begin{proof}
Both claims in the proposition are consequences of the fact that the number of descendent trees of $t$ is equal to the number of leaves of $t$ appearing to the right of all closed vertices.  

Let $v_j$ be the $j$th leaf from the right of these $i$ leaves.  When we append two children to $v_j$, $v_j$'s parent is closed, and this new tree has only $j+1$ descendents.

Further, if $t$ has $i$ leaves to the right of all closed vertices, by definition of closed, none of these leaves' parents have right children that are internal vertices.  The only way to arrange a collection of vertices so that no internal vertices are right children is in the shape of an $i$-leaf left comb.
\end{proof}

Now, consider the descendant relation restricted to trees that avoid the $k$-leaf left comb. Any tree with $k$ descendant trees contains an $k$-leaf comb, so we are only concerned with trees that have at most $k-1$ descendent trees.  As in Proposition \ref{prop:descs}, a tree with $i<k-2$ descendant trees will have descendant trees with $2, 3, 4, \dots , i+1$ descendants respectively. 

Notice further that for a tree which would have $k-1$ descendent trees, one of these descendent trees contains the $k$-leaf left comb, so such a tree only has $k-2$ descendent trees that avoid the $k$-leaf comb.   Consequently, for a tree with $k-2$ descendent trees, those descendents will have $ 2, 3, 4, \dots, k-3, k-2, k-2$ descendent trees respectively that avoid the $k$-leaf left comb.

Let $a_{n,i}$ be the number of $n$-leaf trees that avoid the $k$-leaf left comb and have exactly $i$ $(n+1)$-leaf descendants, and let $\displaystyle{a_n=\sum_{i=1}^{k-2} a_{n,i}}$.

We have the base cases $a_1 = a_{1,1}=1$ and $a_2=a_{2,2}=1$.

More generally, for $n\geq 3$

$$a_{n,2}=a_{n-1,2}+a_{n-1,3}+\cdots +a_{n-1,k-2} = a_{n-1}$$

$$a_{n,3}=a_{n-1,2}+a_{n-1,3}+\cdots +a_{n-1,k-2} = a_{n-1}$$

$$a_{n,i}=a_{n-1,i-1}+a_{n-1,i}+\cdots +a_{n-1,k-2} \text{ for $i<k-2$}$$

$$a_{n,k-2} = a_{n-1,k-3}+2a_{n-1,k-2}$$

Ultimately, we seek a recurrence for $\displaystyle{a_n=\sum_{i=1}^{k-2} a_{n,i}}$.  Adding the above equations and algebraic manipulation produces the result

$$a_n=\sum\limits_{i=2}^{\lfloor \frac{k+1}{2}\rfloor} (-1)^i\binom{k-i}{i-1}a_{n-i+1}.$$

From the discussion above, trees avoiding an $k$-leaf comb can be constructed with a finitely labeled generating tree (in particular with a generating tree using precisely $k-2$ labels), and one may use the transfer matrix method to obtain an alternate proof of Theorem \ref{Th1} for the case of avoiding the $k$-leaf left comb. 

\section{Connections to pattern-avoiding permutations}\label{S:perms}

As evidenced in the Appendix, many sequences obtained by counting trees that avoid non-contiguous binary tree patterns are already known in the literature for other reasons.  In this section we present a theorem that fully explains this connection for the case of avoiding a single binary tree pattern.

To this end, let $S(n)$ denote the set of permutations of length $n$.  As in the introduction, given $\pi \in S(n)$ and $\rho \in S(k)$ we say that $\pi$ contains $\rho$ as a pattern if there exist indices $1 \leq i_1 < \cdots < i_k \leq n$ such that $\pi_{i_a} < \pi_{i_b}$ if and only if $\rho_a < \rho_b$.  Let $S_Q(n)=\{\pi \in S(n) \vert \forall \rho \in Q, \pi \text{ avoids } \rho\}$, and $s_Q(n) = \left|S_Q(n)\right|$.  For example, $s_{\{12\}}(n) = 1$ for $n \geq 1$ since the only way to avoid the pattern 12 is to be the decreasing permutation of length $n$.  It is also well-known that if $\rho \in S(3)$, then $s_{\{\rho\}}(n)=C_n$ where $C_n$ is the $n$th Catalan number.

\begin{theorem}
Let $t$ be any binary tree pattern with $k \geq 2$ leaves.   Then $${\operatorname{av}}_{t}(n) = s_{\{231,(k-1)(k-2)\cdots 21\}}(n-1).$$
\label{permpat}
\end{theorem}

\begin{proof}
It is well known that the set of binary trees with $n$ leaves is in bijection with the set of permutations of length $n-1$ which avoid the pattern 231.

To see this, label the root of tree $t$ with the label $n-1$.  Now, suppose there are $i$ internal vertices to the left of the root and $(n-i-2)$ internal vertices to the right of the root.  The $i$ vertices on the left will receive labels from the set $\{1,\dots,i\}$ and the vertices on the right will receive labels from the set $\{i+1,\dots n-2\}$.  For each subtree, give the root the largest available label and continue recursively until each internal vertex has been labeled.

Now, there is a natural left-to-right ordering of the vertices of $t$.  Read the labels of the vertices from left to right to obtain a permutation $\pi \in S(n-1)$.  Necessarily, $\pi$ avoids $231$ because all labels to the left of a given vertex have smaller labels than all labels to the right.

Further, we can see that the $k$-leaf right comb encodes the unique decreasing permutation of length $k-1$.  It is not hard to see that if a tree avoids the $k$-leaf right comb, then the corresponding permutation avoids the decreasing permutation of length $k-1$ and vice versa.
\end{proof}

We note that this correspondence between $\{231\}$-avoiding permutations and binary trees is not new.  If one ignores the leaves in our trees, the bijection given in the proof of Theorem \ref{permpat} is identical to the correspondence between postorder-labeled trees with inorder-read permutations found in \cite{FHK05}.  Further work connecting permutations to binary trees in the context of sorting can be found in \cite{BM00}, \cite{E79}, \cite{K97}, \cite{R81}, \cite{R78}, and \cite{Stanley99}.   

It is worth considering when Theorem \ref{permpat} generalizes and how.  The correspondence given in the proof of Theorem \ref{permpat} does not necessarily work with trees other than the right comb.  For example, while $\text{av}_{t_{4_1}}(n) = s_{\{231,321\}}(n-1),$ the permutation corresponding to $t_{4_1}$ is 123.  However, $\text{av}_{t_{4_1}}(n)=2^{n-2}$ ($n >1$) while $s_{\{231,123\}}(n-1)=\binom{n-1}{2}+1$ ($n >1$).  Thus even if the sequence obtained from avoiding a set of tree patterns is also the avoidance sequence for some set of pattern-avoiding permutations, there may be other bijections required to demonstrate the equivalence.  It remains open to give a natural interpretation of Theorem \ref{permpat} for trees other than the right comb.

\section{Avoiding two or more binary trees simultaneously}\label{S:multipatterns}

Thus far, we have analyzed generating functions and enumeration sequences for trees avoiding a single non-contiguous binary tree pattern. We will now consider trees that avoid two tree patterns simultaneously. To this end, let $g_{(\{t_{1}, t_{2}\};p)}(x)$ be the generating function that enumerates trees avoiding $t_1$ and $t_2$ with pattern $p$ at the root according to number of leaves.  For brevity, we may abbreviate $g_{(\{t_{1}, t_{2}\};\st{1}{1})}(x)=g_{\{t_{1}, t_{2}\}}(x)$.

Parallel to Equations \ref{eq:t11} and \ref{eq:t21} we have

\begin{equation}
	g_{(\{t_{1}, t_{2}\};\st{1}{1})}(x)= x + g_{(\{t_{1}, t_{2}\};\st{2}{1})}(x)
	\label{eq:2treebase}
\end{equation} 

and

\begin{equation}
	\begin{aligned}
 g_{(\{t_{1}, t_{2}\};\st{2}{1})}(x)&= g_{(\{t_{1_{\ell}}, t_{2_\ell}\};\st{1}{1})}(x) g_{(\{t_{1}, t_{2}\};\st{1}{1})}(x) + g_{(\{t_{1}, t_{2}\};\st{1}{1})}(x) \cdot g_{(\{t_{1_r}, t_{2_r}\};\st{1}{1})}(x) \\
&+ g_{(\{t_{1_\ell}, t_{2}\};\st{1}{1})}(x)\cdot g_{(\{t_{1}, t_{2_r}\};\st{1}{1})}(x)+ g_{(\{t_{1}, t_{2_\ell}\};\st{1}{1})}(x)\cdot g_{(\{t_{1_r}, t_{2}\};\st{1}{1})}(x) \\
&- g_{(\{t_{1}, t_{2_\ell}\};\st{1}{1})}(x) \cdot g_{(\{t_{1_r}, t_{2_r}\};\st{1}{1})}(x) - g_{(\{t_{1_\ell}, t_{2}\};\st{1}{1})}(x) \cdot g_{(\{t_{1_r}, t_{2_r}\};\st{1}{1})}(x) \\
& - g_{(\{t_{1_\ell}, t_{2_\ell}\};\st{1}{1})}(x)\cdot g_{(\{t_{1}, t_{2_r}\};\st{1}{1})}(x) - g_{(\{t_{1_\ell}, t_{2_\ell}\};\st{1}{1})}(x)\cdot g_{(\{t_{1_r}, t_{2}\};\st{1}{1})}(x) \\
& +g_{(\{t_{1_\ell}, t_{2_\ell}\};\st{1}{1})}(x) \cdot g_{(\{t_{1_r}, t_{2_r}\};\st{1}{1})}(x). \\
	\end{aligned}
\label{eq:2treerec}
\end{equation}

This latter expression follows again from inclusion-exclusion.  One can solve this pair of equations for $g_{\{t_1,t_2\}}(x)$ to obtain a rational combination of rational generating functions and see that for any pair of binary trees, $g_{\{t_1,t_2\}}(x)$ is indeed rational.

To compactly state the equivalent expression to Equation \ref{eq:2treerec} for trees avoiding $j$ binary tree patterns ($j \in \mathbb{Z}^{+}$), we need to introduce more efficient notation.  Notice that all terms on the right hand sides of Equations \ref{eq:t21} and \ref{eq:2treerec} are of the form $g_{(S_\ell;\st{1}{1})}(x) \cdot g_{(S_r;\st{1}{1})}(x)$ where for each tree $t_i$ $(1 \leq i \leq j)$, there are 3 possibilities: (i) $t_i\in S_\ell$ and $t_{i_r} \in S_r$, (ii) $t_{i_\ell} \in S_\ell$ and $t_{i} \in S_r$, or (iii) $t_{i_\ell} \in S_\ell$ and $t_{i_r} \in S_r$.  Let $v \in \{-1,0,1\}^j$ and let $gp_{\vec{v}}(x)=g_{(S_\ell;\st{1}{1})}(x) \cdot g_{(S_r;\st{1}{1})}(x)$ where (i) if $v_i=-1$ then $t_i\in S_\ell$ and $t_{i_r} \in S_r$, (ii) if $v_i=1$ then $t_{i_\ell} \in S_\ell$ and $t_{i} \in S_r$, and (iii) if $v_i=0$ then $t_{i_\ell} \in S_\ell$ and $t_{i_r} \in S_r$.  Further, for $\vec{v} \in \{-1,0,1\}^j$, define $\left|\vec{v}\right|=\sum_{i=1}^j \left|v_i\right|$.

By inclusion-exclusion we have:

\begin{equation}
g_{(\{t_{1}, \dots, t_{j}\};\st{2}{1})}(x) = \sum_{\vec{v} \in \{-1,0,1\}^j} (-1)^{j-\left|v\right|} gp_{\vec{v}}(x).
\end{equation}

Notice that this expression is linear in $g_{(\{t_{1}, \dots, t_{j}\};\st{1}{1})}(x)$.  In fact, $g_{(\{t_{1}, \dots, t_{j}\};\st{1}{1})}(x)$ only appears in two terms: the ones corresponding to $\vec{v}=\langle 1,\dots, 1\rangle$ and $\vec{v}=\langle-1,\dots,-1\rangle$, so we see that $g_{\{t_1,\dots t_k\}}(x)$ is rational for any finite set of $j$ tree patterns.

We used Equations \ref{eq:2treebase} and \ref{eq:2treerec} to compute $g_{\{t_1,t_2\}}(x)$ for any pair of binary tree patterns where $t_1$ and $t_2$ have no more than 5 leaves.  A summary of these results is given in the Appendix.  

In light of Theorem \ref{permpat} one might wonder if for every set of binary tree patterns $S$, $\{\text{av}_{S}(n)\}_{n=2}^\infty$ is identical to the avoidance sequence $\{s_{Q}(n)\}_{n=1}^\infty$ for some set $Q$ of permutation patterns.  This, however turns out not to be the case.  For example consider the sequence 

$$\left\{\text{av}_S(n)\right\}_{n=2}^\infty = 1, 2, 5, 12, 26, 49, 83, 129, \dots \text{ where } S={\{\st{5}{6},\st{5}{14}\}}.$$  If this were the avoidance sequence for some set of permutation patterns $Q$, we see from the fact that $a_3=5$ that $Q$ contains a pattern of length 3.  Further, since it is known that $s_4(\rho)=14$ for any pattern $\rho \in S(3)$, $Q$ must contain precisely two patterns of length 4.  However, exhaustive checking of $\{s_n(Q)\}_{n=1}^8$ for sets of patterns $Q$ consisting of one pattern of length 3 and two patterns of length 4 yields no match for this sequence.

\section{Conclusion}

Throughout this paper, we have investigated non-contiguous pattern avoidance in binary trees.  Unlike Rowland's work for contiguous patterns, our avoidance generating functions are always rational, and the Wilf classes obtained  for avoidance of single patterns are easy to describe: two tree patterns are Wilf-equivalent if and only if they have the same number of leaves.  The results in this paper not only give a complete characterization of trees avoiding a single pattern, but we also present a computational technique to quickly determine the number of trees avoiding any finite set of non-contiguous tree patterns.  Finally we explore combinatorial proofs of our results and give an explicit bijection between certain pattern-avoiding trees and pattern-avoiding permutations.

\section*{Acknowledgement}

The authors thank Eric Rowland for assistance with generating the tree graphics for this paper and for several helpful presentation suggestions.

\section*{Appendix}

This appendix lists pairs of binary tree patterns each of which have at most 5 leaves, classifying them by their avoidance generating function and sequence.  For each class, we give the generating function $g_{\{t_{1},t_{2}\}}(x)$, and we list the first 15 terms of the corresponding avoidance sequence.  If the avoidance sequence for a class is listed in the Online Encyclopedia of Integer Sequences \cite{OEIS}, we include the appropriate reference.  For brevity, left--right reflections are omitted.

\subsection*{Avoiding a 3-Leaf \& a 4-Leaf Tree}

\begin{center} 
\textit{Class A}
\end{center}
\begin{itemize}
\item $g_{\{t_{1},t_{2}\}}(x)= x+x^{2}+x^{3}$ 
\item Sequence: $1, 1, 1, 0, 0, 0, 0, 0, 0, 0, 0, 0, 0, 0, 0, \dots$
\end{itemize}

\begin{center}
$\left\{\st{3}{1},\st{4}{5}\right\}$
\end{center}

\hrule

\subsection*{Avoiding a 3-Leaf \& a 5-Leaf Tree}

\begin{center} 
\textit{Class A}
\end{center}
\begin{itemize}
\item $g_{\{t_{1},t_{2}\}}(x)= x+x^{2}+x^{3}+x^{4}$ 
\item Sequence: $1, 1, 1, 1, 0, 0, 0, 0, 0, 0, 0, 0, 0, 0, 0,\dots$  
\end{itemize}

\begin{center}
$\left\{\st{3}{1},\st{5}{14} \right\}$
\end{center}

\hrule

\subsection*{Avoiding a 4-Leaf \& a 5-Leaf Tree}

\begin{center} 
\textit{Class A}
\end{center}
\begin{itemize}
\item $g_{\{t_{1},t_{2}\}}(x)= x+x^{2}+2x^{3}+4x^{4} + 7x^{5}+ 8x^{6}+8x^{7}+6x^{8}+3x^{9}+x^{10}$ 
\item Sequence: $1, 1, 2, 4, 7, 8, 8, 6, 3, 1, 0, 0, 0, 0, 0, \dots$
\end{itemize}

\begin{center}
$\left\{\st{4}{1},\st{5}{14} \right\}$
\end{center}

\hrule

\begin{center} 
\textit{Class B}
\end{center}

\begin{itemize}
\item $g_{\{t_{1},t_{2}\}}(x)= \frac{x-x^{2}+x^{3}+x^{4}+x^{5}}{1-2x+x^{2}}$ 
\item Sequence: $1, 1, 2, 4, 7, 10, 13, 16, 19, 22, 25, 28, 31, 34, 37,  \dots $
\item OEIS A016777: $3k+1$ for $k\geq 4.$ 
\end{itemize}

\begin{center}
$\left\{\st{4}{1},\st{5}{9} \right\}$
\end{center}

\hrule

\begin{center} 
\textit{Class C}
\end{center}
\begin{itemize}
\item $g_{\{t_{1},t_{2}\}}(x)= \frac{-x+2x^{2}-2x^{3}}{-1+3x-3x^{2}+x^{3}}$ 
\item Sequence: $1, 1, 2, 4, 7, 11, 16, 22, 29, 37, 46, 56, 67, 79, 92,  \dots$
\item OEIS A152947: $\frac{(k-2)\cdot (k-1)+1}{2}$	
\end{itemize}

\begin{center}
$\left\{\st{4}{1},\st{5}{13} \right\}, \left\{\st{4}{1},\st{5}{8} \right\}, \left\{\st{4}{1},\st{5}{12} \right\}, \left\{\st{4}{2},\st{5}{13} \right\},$\\
$\left\{\st{4}{2},\st{5}{14} \right\}, \left\{\st{4}{2},\st{5}{8} \right\}, \left\{\st{4}{2},\st{5}{9} \right\}, \left\{\st{4}{2},\st{5}{12} \right\},$\\
$\left\{\st{4}{3},\st{5}{1} \right\}, \left\{\st{4}{3},\st{5}{2} \right\}, \left\{\st{4}{3},\st{5}{4} \right\}, \left\{\st{4}{3},\st{5}{5} \right\}$
\end{center}

\hrule

\begin{center} 
\textit{Class D}
\end{center}
\begin{itemize}
\item $g_{\{t_{1},t_{2}\}}(x)= \frac{x-x^{2}+x^{4}}{1-2x+x^{3}}$ 
\item  Sequence: $1, 1, 2, 4, 7, 12, 20, 33, 54, 88, 143, 232, 376, 609, 986, \dots$
\item OEIS A000071: Fibonacci numbers -1 for $n \geq 2.$
\end{itemize}
\begin{center}
$\left\{\st{4}{1},\st{5}{7} \right\}, \left\{\st{4}{1},\st{5}{11} \right\}, \left\{\st{4}{2},\st{5}{10} \right\}, \left\{\st{4}{2},\st{5}{6} \right\}$
\end{center}

\hrule

\begin{center}
\textit{Class E}
\end{center}
\begin{itemize}
\item $g_{\{t_{1},t_{2}\}}(x)= \frac{-x}{-1+x+x^{2}+x^{3}}$ 
\item Sequence: $1, 1, 2, 4, 7, 13, 24, 44, 81, 149, 274, 504, 927, 1705, 3136,  \dots$
\item OEIS A000073: Tribonacci Numbers 
\end{itemize}
\begin{center}
$\left\{\st{4}{1},\st{5}{5} \right\}, \left\{\st{4}{2},\st{5}{1} \right\}$
\end{center}

\hrule

\subsection*{Avoiding a Pair of 4 Leaf Trees}

\begin{center} 
\textit{Class A}
\end{center}
\begin{itemize} 
\item $g_{\{t_{1},t_{2}\}}(x)= x+x^{2}+2x^{3}+3x^{4} + 2x^{5}+ x^{6} $
\item Sequence: $1, 1, 2, 3, 2, 1, 0, 0, 0, 0, 0, 0, 0, 0, 0, \dots$
\end{itemize}

\begin{center}
$\left\{\st{4}{1},\st{4}{5} \right\}$
\end{center}

\hrule

\begin{center} 
\textit{Class B}
\end{center}
\begin{itemize}
\item $g_{\{t_{1},t_{2}\}}(x)= \frac{x-x^{2}+x^{3}}{1-2x+x^{2}}$ 
\item Sequence: $1, 1, 2, 3, 4, 5, 6, 7, 8, 9, 10, 11, 12, 13, 14,  \dots$
\item OEIS A000027: The natural numbers
\end{itemize}

\begin{center}
$\left\{\st{4}{1},\st{4}{3} \right\}, \left\{\st{4}{1},\st{4}{4} \right\}, \left\{\st{4}{2},\st{4}{3} \right\}, \left\{\st{4}{2},\st{4}{4} \right\}$
\end{center}

\hrule

\begin{center} 
\textit{Class C}
\end{center}
\begin{itemize}
\item $g_{\{t_{1},t_{2}\}}(x)= \frac{-x}{-1+x+x^{2}}$ 
\item Sequence: $1, 1, 2, 3, 5, 8, 13, 21, 34, 55, 89, 144, 233, 377, 610,  \dots$
\item OEIS A000045: Fibonacci Numbers 
\end{itemize}
\begin{center}
$\left\{\st{4}{1},\st{4}{2} \right\}$
\end{center}

\hrule

\subsection*{Avoiding a Pair of 5 Leaf Trees}
\textbf{Note:} The first five terms of the sequences in this section will be 1, 1, 2, 5, 12. Therefore, the sequences will begin with the sixth term. 

\begin{center}
\textit{Class A}
\end{center}
\begin{itemize}
\item	$g_{\{t_{1},t_{2}\}}(x)=  x+ x^{2}+2x^{3}+5x^{4}+12x^{5}+26x^{6}+46x^{7}+76x^{8}\\ 
					   \phantom{g_{\{t_{1},t_{2}\}}(x)=} +116x^{9}+163x^{10} + 208x^{11}+ 238x^{12}+ 240x^{13}+210x^{14} \\
					   \phantom{g_{\{t_{1},t_{2}\}}(x)=} +158x^{15}+ 100x^{15}+52x^{17}+21x^{18}+6x^{19}+x^{20} $
\item Sequence: $26, 46, 76, 116, 163, 208, 238, 240, 210, 158, 100, 52, 21, 6 , 1$
\end{itemize}
\begin{center}
$\left\{\st{5}{1},\st{5}{14} \right\}$
\end{center}

\hrule

\begin{center} 
\textit{Class B}
\end{center}
\begin{itemize}
\item $g_{\{t_{1},t_{2}\}}(x)= \frac{x-2x^{2}+2x^{3}+x^{4}+2x^{5}+3x^{6}+2x^{7}+2x^{8}+x^{9}}{1-3x+3x^{2}-x^{3}}$
\item Sequence: $26, 49, 83, 129, 187, 257, 339, 433, 539, 657,  \dots$
\item New to OEIS
\end{itemize}
\begin{center}
$\left\{\st{5}{6},\st{5}{14} \right\}$
\end{center}

\hrule

\begin{center} 
\textit{Class C}
\end{center}
\begin{itemize}
\item $g_{\{t_{1},t_{2}\}}(x)= \frac{x-4x^{2}+7x^{3}-5x^{4}+2x^{5}}{1-5x+10x^{2}-10x^{3}+5x^{4}-x^{5}}$ 
\item  Sequence: $26, 51, 92, 155, 247, 376, 551, 782, 1080, 1457,  \dots$
\item OEIS A027927: $T(k,2k-4), T$ given by A027926 for $n \geq 2.$
\end{itemize}
\begin{center}
$\left\{\st{5}{1},\st{5}{8} \right\}, \left\{\st{5}{2},\st{5}{8} \right\}, \left\{\st{5}{2},\st{5}{9} \right\}$
\end{center}

\hrule

\begin{center}
\textit{Class D}
\end{center}
\begin{itemize} 
\item $g_{\{t_{1},t_{2}\}}(x)= \frac{x-5x^{2}+11x^{3}-12x^{4}+7x^{5}-2x^{6}+x^{7}}{1-6x+15x^{2}-20x^{3}+15x^{4}-6x^{5}+x^{6}}$ 
\item Sequence: $26, 52, 98, 176, 303, 502, 803, 1244, 1872, 2744, \dots$
\item New to OEIS
\end{itemize}

\begin{center}
$\left\{\st{5}{1},\st{5}{12} \right\}, \left\{\st{5}{1},\st{5}{13} \right\}, \left\{\st{5}{2},\st{5}{12} \right\}, \left\{\st{5}{2},\st{5}{13} \right\}$
\end{center}

\hrule

\begin{center} 
\textit{Class E}
\end{center}
\begin{itemize}
\item $g_{\{t_{1},t_{2}\}}(x)= \frac{x-3x^{2}+3x^{3}+x^{4}-x^{5}}{1-4x+5x^{2}-x^{3}-2x^{4}+x^{5}}$ 
\item Sequence: $26, 52, 98, 177, 310, 531, 895, 1491, 2463, 4044,  \dots$
\item OEIS A116717: Number of permutations of length k which avoid the patterns 231, 1423, 3214 for $n \geq 2.$
\end{itemize}
\begin{center}
$\left\{\st{5}{4},\st{5}{9} \right\}, \left\{\st{5}{7},\st{5}{10} \right\}$
\end{center}

\hrule

\begin{center} 
\textit{Class F}
\end{center}
\begin{itemize}
\item $g_{\{t_{1},t_{2}\}}(x)= \frac{-x+2x^{2}-x^{3}-x^{4}-2x^{5}}{-1+3x-2x^{2}-x^{4}+x^{5}}$
\item Sequence: $26, 53, 104, 199, 375, 700, 1299, 2402, 4432, 8167,  \dots$
\item New to OEIS
\end{itemize}
\begin{center}
$\left\{\st{5}{5},\st{5}{6} \right\}$
\end{center}

\hrule

\begin{center}
\textit{Class G}
\end{center}
\begin{itemize}
\item $g_{\{t_{1},t_{2}\}}(x)= \frac{x-2x^{2}+2x^{4}+2x^{5}-x^{6}-x^{7}}{1-3x+x^{2}+2x^{3}+x^{4}-x^{5}-x^{6}}$ 
\item Sequence: $26, 55, 113, 227, 449, 877, 1696, 3254, 6203, 11762, \dots$
\item OEIS A116726: Number of permutations of length $k$ which avoid the patterns 213, 1234, 2431 for $n \geq 2.$
\end{itemize}
\begin{center}
$\left\{\st{5}{1},\st{5}{7} \right\},\left\{\st{5}{1},\st{5}{11} \right\}$
\end{center}

\hrule

\begin{center} 
\textit{Class H}
\end{center}
\begin{itemize}
\item $g_{\{t_{1},t_{2}\}}(x)= \frac{x-x^{2}-x^{3}+3x^{5}+2x^{6}+x^{7}}{1-2x-x^{2}+3x^{4}+2x^{5}+x^{6}}$ 
\item Sequence: $26, 56, 118, 244, 499, 1010, 2027, 4040, 8004, 15776,  \dots$
\item OEIS A073778: $a(m)=\sum_{k=0}^{m}{T(k)\cdot T(m-k)}.$ Convolution of tribonacci sequence A000073 with itself for $m \geq 3$, for $n \geq 2.$ 
\end{itemize}
\begin{center}
$\left\{\st{5}{5},\st{5}{10} \right\}$
\end{center}

\hrule

\begin{center}
\textit{Class I}
\end{center}
\begin{itemize}
\item $g_{\{t_{1},t_{2}\}}(x)= \frac{-x}{-1+x+x^{2}+2x^{3}+3x^{4}+2x^{5}+x^{6}}$ 
\item Sequence: $26, 57, 127, 284, 632, 1405, 3126, 6958, 15485, 34458,  \dots$
\item New to OEIS 
\end{itemize}

\begin{center}
$\left\{\st{5}{1},\st{5}{5} \right\}$
\end{center}

\hrule

\begin{center} 
\textit{Class J}
\end{center}
\begin{itemize}
\item $g_{\{t_{1},t_{2}\}}(x)= \frac{-x+3x^{2}-3x^{3}}{-1+4x-5x^{2}+2x^{3}}$ 
\item Sequence: $27, 58, 121, 248, 503, 1014, 2037, 4084, 8179, 16370, \dots$
\item OEIS A000325: $2^{k} - k$ 
\end{itemize}
\begin{center}
$\left\{\st{5}{4},\st{5}{6} \right\}, \left\{\st{5}{4},\st{5}{7} \right\}, \left\{\st{5}{4},\st{5}{8} \right\}, \left\{\st{5}{5},\st{5}{7} \right\},$\\
$\left\{\st{5}{5},\st{5}{12} \right\}, \left\{\st{5}{5},\st{5}{14} \right\}, \left\{\st{5}{6},\st{5}{8} \right\}, \left\{\st{5}{6},\st{5}{9} \right\},$\\
$\left\{\st{5}{7},\st{5}{12} \right\}, \left\{\st{5}{3},\st{5}{10} \right\}, \left\{\st{5}{3},\st{5}{11} \right\}, \left\{\st{5}{3},\st{5}{8} \right\},$\\
$\left\{\st{5}{3},\st{5}{9} \right\}, \left\{\st{5}{6},\st{5}{10} \right\}, \left\{\st{5}{7},\st{5}{8} \right\}$\\
\end{center}

\hrule

\begin{center} 
\textit{Class K}
\end{center}
\begin{itemize}
\item $g_{\{t_{1},t_{2}\}}(x)= \frac{x-2x^{2}+2x^{4}+x^{5}}{1-3x+x^{2}+2x^{3}}$ 
\item Sequence: $27, 59, 126, 263, 551, 1136, 2327, 4743, 9630, 19493, \dots$
\item OEIS A116712: Number of permutations of length k which avoid the patterns 231, 3214, 4312  for $n \geq 2.$ 
\end{itemize}
\begin{center}
$\left\{\st{5}{2},\st{5}{6} \right\}, \left\{\st{5}{2},\st{5}{10} \right\}$\\
\end{center}

\hrule

\begin{center}
\textit{Class L}
\end{center}
\begin{itemize}
\item $g_{\{t_{1},t_{2}\}}(x)= \frac{x-3x^{2}+2x^{3}+x^{4}}{1-4x+4x^{2}}$ 
\item Sequence: $28, 64, 144, 320, 704, 1536, 3328, 7168, 15360, 32768,  \dots$
\item OEIS A045623: Number of 1's in all compositions of $k+1$ for $n \geq 2.$ 
\end{itemize}
\begin{center}
$\left\{\st{5}{1},\st{5}{6} \right\}, \left\{\st{5}{1},\st{5}{10} \right\}, \left\{\st{5}{2},\st{5}{7} \right\}, \left\{\st{5}{2},\st{5}{11} \right\}$\\
$\left\{\st{5}{3},\st{5}{6} \right\}, \left\{\st{5}{3},\st{5}{7} \right\}, \left\{\st{5}{3},\st{5}{12} \right\}, \left\{\st{5}{4},\st{5}{10} \right\}$\\
$\left\{\st{5}{4},\st{5}{11} \right\}, \left\{\st{5}{4},\st{5}{13} \right\}$\\
\end{center}

\hrule

\begin{center} 
\textit{Class M}
\end{center}
\begin{itemize}
\item $g_{\{t_{1},t_{2}\}}(x)= \frac{x-2x^{2}+x^{3}}{1-3x+2x^{2}-x^{3}}$ 
\item Sequence: $28, 65, 151, 351, 816, 1897, 4410, 10252, 23833, 55405,  \dots$
\item OEIS A034943: Binomial transform of Padovan sequence A000931 for $n\geq 1.$
\end{itemize}
\begin{center}
$\left\{\st{5}{1},\st{5}{3} \right\}, \left\{\st{5}{1},\st{5}{4} \right\}, \left\{\st{5}{2},\st{5}{3} \right\}, \left\{\st{5}{2},\st{5}{4} \right\}$\\
$\left\{\st{5}{2},\st{5}{5} \right\}, \left\{\st{5}{3},\st{5}{4} \right\}, \left\{\st{5}{3},\st{5}{5} \right\}$\\
\end{center}

\hrule

\begin{center} 
\textit{Class N}
\end{center}
\begin{itemize}
\item $g_{\{t_{1},t_{2}\}}(x)= \frac{x-x^{2}-x^{3}}{1-2x-x^{2}}$ 
\item Sequence: $29, 70, 169, 408, 985, 2378, 5741, 13860, 33461, 80782,  \dots$
\item OEIS A000129: Pell numbers: $a(0) = 0, a(1) = 1;$ for $k \geq 1$, $a(k) = 2 \cdot a(k-1) + a(k-2)$ for $n \geq 2.$ 
\end{itemize}
\begin{center}
$\left\{\st{5}{1},\st{5}{2} \right\}, \left\{\st{5}{4},\st{5}{5} \right\}, \left\{\st{5}{6},\st{5}{9} \right\}$\\
\end{center}

\hrule

\end{document}